\documentclass[12pt]{amsart}

\usepackage{amsthm,amssymb,enumitem,mathtools,tabularx,graphicx}

\usepackage{tikz}

\usepackage{scalerel}
\usepackage{hyperref}

\usetikzlibrary[calc,intersections,through,backgrounds,arrows,decorations.pathmorphing]

\usetikzlibrary{arrows,shapes,automata,backgrounds,decorations,petri,positioning}
\tikzstyle{edge} = [fill,opacity=.5,fill opacity=.5,line cap=round, line join=round, line width=50pt]

\pgfdeclarelayer{background}
\pgfsetlayers{background,main}

\usepackage[margin=1in,letterpaper,portrait]{geometry}

\usepackage{cite}

\usepackage[latin1]{inputenc}
\usepackage{subfigure}
\usepackage{color}
\usepackage{amsmath}
\usepackage{amsthm}
\usepackage{amstext}
\usepackage{amssymb}
\usepackage{amsfonts}
\usepackage{graphicx}
\usepackage{young}
\usepackage{multicol}
\usepackage{mathrsfs}
\usepackage[all]{xy}
\usepackage{tikz}
\usepackage{mathtools}
\usepackage{tabularx}
\usepackage{array}
\usepackage{commath}
\usepackage{ytableau}

\theoremstyle{plain}
\theoremstyle{definition}
\newtheorem{theorem}{Theorem}[section]

\newtheorem{lemma}[theorem]{Lemma}

\newtheorem{definition}[theorem]{Definition}

\newtheorem{example}[theorem]{Example}
\newtheorem{proposition}[theorem]{Proposition}
\newtheorem{corollary}[theorem]{Corollary}

\DeclareMathAlphabet{\mathpzc}{OT1}{pzc}{m}{it}

\newcommand{\mf}[1]{\mbox{$\mathfrak #1$}}

\begin{document}

\title{Forced perimeter in Elnitksy polygons}

\author{Bridget Eileen Tenner}
\address{Department of Mathematical Sciences, DePaul University, Chicago, IL, USA}
\email{bridget@math.depaul.edu}
\thanks{Research partially supported by a DePaul University Faculty Summer Research Grant and by Simons Foundation Collaboration Grant for Mathematicians 277603.}



\begin{abstract}
We study tiling-based perimeter and characterize when a given perimeter tile appears in all rhombic tilings of an Elnitsky polygon. Regardless of where on the perimeter this tile appears, its forcing can be described in terms of $321$-patterns. We characterize the permutations with maximally many forced right-perimeter tiles, and show that they are enumerated by the Catalan numbers.
\end{abstract}

\maketitle

\section{Introduction}\label{sec:intro}

In recent work, we developed tiling-based notions of perimeter and area \cite{tenner iso-tiling}. One focus of that work was on rhombic tilings of Elnitsky polygons, due to their significance for the combinatorics of reduced decompositions of permutations. This paper continues that work through the notion of a ``forced'' perimeter tile. In addition to the combinatorial significance of such an object, this has an important analogue in the traditional use of isoperimetric quotients to assess geographic compactness for legal and electoral purposes. In those settings, excessive perimeter can be penalizing, despite the fact that some regions---such as those on a coast---necessarily have extensive perimeter, through no manipulative endeavors. (See \cite{duchin tenner} for more background on the subject.)

The combinatorial significance of perimeter tiles was discussed in \cite[Corollary 2.17]{tenner iso-tiling}. Here we take that relevance as a given, providing motivation for studying these objects, and we devote the rest of this work to an analysis of when perimeter tiles are ``forced.''

The results that we present here tie together reduced decompositions, pattern avoidance, and Catalan numbers. We begin with a brief overview of Elnitsky polygons in Section~\ref{sec:elnitsky} and perimeter tiles in Section~\ref{sec:perimeter}. In Section~\ref{sec:forced}, we introduce the notion of a ``forced'' perimeter tile and characterize the conditions under which a given right-, left-, top-, or bottom-perimeter tile is forced (Theorem~\ref{thm:force right}, Corollary~\ref{cor:force left}, Theorem~\ref{thm:force top}, and Corollary~\ref{cor:force bottom}, respectively). Each of these results can be phrased in terms of $321$-patterns. Having established those conditions, we use Section~\ref{sec:optimizing} to look at the special case of permutations with maximally forced right-perimeter tiles. These are characterized in Theorem~\ref{thm:descents in optimal}, and Theorem~\ref{thm:optimal right enumeration} shows that they are enumerated by Catalan numbers. We conclude the paper with suggestions for some of the many directions in which to extend this research.

\section{Elnitsky polygons}\label{sec:elnitsky}

We briefly introduce the primary object of this work, assuming a familiarity with permutations, simple reflections, reduced decompositions, commutation and braid relations, inversions, and so on. For further background and a hint at the deep and varied mathematical interest in these topics, the reader is referred to \cite{bedard, bjs, bjorner brenti, elnitsky, fishel milicevic patrias tenner, jonsson welker, stanley, tenner rdpp, tenner rwm, zollinger}.

In \cite{elnitsky}, Elnitsky related the commutation classes of reduced decompositions of a permutation $w$ to rhombic tilings of a particular polygon $X(w)$, and the bijection that he developed can be exploited to elucidate properties of permutations. It also gives combinatorial meaning to the tiles appearing in such a tiling.

\begin{definition}
Fix $w \in \mf{S}_n$. \emph{Elnitsky's polygon} for $w$ is the equilateral $2n$-gon defined so that:
\begin{itemize}
\item sides are labeled $1, \ldots, n, w(n), \ldots, w(1)$ in counterclockwise order,
\item the first $n$ of those sides form half of a convex $2n$-gon, and
\item sides with the same label are parallel.
\end{itemize}
We refer to this $2n$-gon as $X(w)$.
\end{definition}

The tilings that we consider---and the tilings that Elnitsky permits---follow certain rules.

\begin{definition}
A \emph{rhombic tiling} of $X(w)$ consists of tiles whose edges are all congruent and parallel to the edges of $X(w)$. The set of all rhombic tilings of $X(w)$ is $T(w)$. The \emph{labels} of a tile $t$ appearing in some $T \in T(w)$ are the labels of the sides of $X(w)$ to which the edges of $t$ are parallel.
\end{definition}

\begin{corollary}[see {\cite[Corollary 2.12]{tenner iso-tiling}}]\label{cor:labels}
In any rhombic tiling of an Elnitsky polygon, no two tiles have the same labels. Moreover, there exists a tile labeled $\{x<y\}$ if and only if $w^{-1}(x) > w^{-1}(y)$.
\end{corollary}

Due to this result, it will cause no confusion if we refer to a tile by its labels, writing ``the tile $\{x,y\}$'' instead of ``the tile labeled $\{x,y\}$.''

Because we are interested in tilings, it suffices to consider tiling regions that are contiguous. Thus, throughout this paper, each permutation $w \in \mf{S}_n$ will be assumed to satisfy
$$\{w(1),\ldots,w(r)\} \neq \{1,\ldots,r\}$$
for all $r < n$. A permutation $w \in \mf{S}_n$ is \emph{fully supported} if all simple reflections appear in the reduced decompositions for $w$. Therefore, as discussed in \cite{tenner repetition},  another way of phrasing this assumption is to say that we assume all permutations are fully supported.

For the sake of consistency, we orient Elnitsky polygons---in figures and for the sake of discussion---so that the top vertex is the intersection of the sides labeled $1$ and $w(1)$, and the counterclockwise path from the top vertex to the bottom vertex is the \emph{leftside boundary}. We mark the top and bottom vertices with dots.

\begin{example}
The Elnitsky polygon $X(34251)$ is illustrated in Figure~\ref{fig:X(34251)}.
\end{example}

\begin{figure}[htbp]
\begin{tikzpicture}
\begin{scope}
\clip (.2,2) rectangle (-2,-2);
\node[draw=none,minimum size=4cm,regular polygon,regular polygon sides=10] (a) {};
\end{scope}
\foreach \x in {1,2,3,4,5,6} {\coordinate (corner \x) at (a.side \x);}
\foreach \y [evaluate=\y as \x using \y+1] in {1,2,3,4,5} {\coordinate (side \y) at ($(corner \x)-(corner \y)$);}
\foreach \y [evaluate=\y as \x using \y+1] in {1,2,3,4,5} {\coordinate (side -\y) at ($(corner \y)-(corner \x)$);}
\foreach \x in {(a.side 1),(a.side 6)} {\fill \x circle (2pt);}
\draw (corner 1) -- (corner 2) -- (corner 3) -- (corner 4) -- (corner 5) -- (corner 6);
\draw (corner 1) --++(side 3) --++(side 4) --++(side 2) --++(side 5) --++(side 1);
\end{tikzpicture}
\caption{The Elnitsky polygon for $34251$.}
\label{fig:X(34251)}
\end{figure}
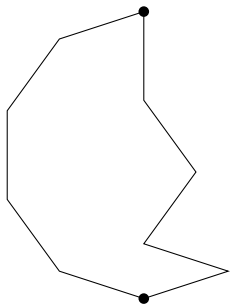

\section{Perimeter tiles}\label{sec:perimeter}

In \cite{tenner iso-tiling}, we introduced the idea of a ``perimeter'' tile. We reiterate that here, in the context of rhombic tilings of Elnitsky polygons.

\begin{definition}
Fix a permutation $w$ and a rhombic tiling $T \in T(w)$. A tile $t$ that shares at least two consecutive edges with the boundary of $X(w)$ is a \emph{perimeter tile}. We can further specify whether $t$ is a \emph{left-}, \emph{right-}, \emph{top-} or \emph{bottom-perimeter} tile based on whether it shares two consecutive edges with the leftside boundary of $X(w)$, the rightside boundary of $X(w)$, the edges on either side of the top vertex of $X(w)$, or the edges of either side of the bottom vertex of $X(w)$, respectively. This specification is the \emph{type} of the tile.
\end{definition}

Note that a perimeter tile may have more than one type. As a trivial example, the sole tile in $T(21)$ is a left-, right-, top-, and bottom-perimeter tile.

The perimeter tiles that appear among all elements of $T(w)$ can vary. Perimeter properties for elements of $T(n\cdots321)$ were studied in \cite{tenner iso-tiling}, as were the permutations $w$ for which elements of $T(w)$ have minimally many perimeter tiles. In the present work, we consider perimeter tiles from a different perspective: namely, when a given perimeter tile appears among all elements of $T(w)$.

\section{Forced tiles}\label{sec:forced}

Fix a permutation $w$ and consider the rhombic tilings $T(w)$. The perimeter tiles that appear may vary between tilings, as they do in the two rhombic tilings of $X(321)$ depicted in Figure~\ref{fig:321}.

\begin{figure}[htbp]
\begin{tikzpicture}
\begin{scope}
\clip (.2,1.2) rectangle (-1.2,-1.2);
\node[draw=none,minimum size=2.4cm,regular polygon,regular polygon sides=6] (a) {};
\foreach \x in {(a.side 1),(a.side 4)} {\fill \x circle (2pt);}
\end{scope}
\foreach \x in {1,2,3,4} {\coordinate (corner \x) at (a.side \x);}
\foreach \y [evaluate=\y as \x using \y+1] in {1,2,3} {\coordinate (side \y) at ($(corner \x)-(corner \y)$);}
\foreach \y [evaluate=\y as \x using \y+1] in {1,2,3} {\coordinate (side -\y) at ($(corner \y)-(corner \x)$);}
\draw (corner 1) -- (corner 2) -- (corner 3) -- (corner 4);
\draw (corner 1) --++(side 3) coordinate (corner 6) --++(side 2) coordinate (corner 5) --++(side 1);
\draw (corner 6) --++(side 1) --++(side -3);
\draw (corner 4) --++(side -2);
\end{tikzpicture}
\hspace{.5in}
\begin{tikzpicture}
\begin{scope}
\clip (.2,1.2) rectangle (-1.2,-1.2);
\node[draw=none,minimum size=2.4cm,regular polygon,regular polygon sides=6] (a) {};
\foreach \x in {(a.side 1),(a.side 4)} {\fill \x circle (2pt);}
\end{scope}
\foreach \x in {1,2,3,4} {\coordinate (corner \x) at (a.side \x);}
\foreach \y [evaluate=\y as \x using \y+1] in {1,2,3} {\coordinate (side \y) at ($(corner \x)-(corner \y)$);}
\foreach \y [evaluate=\y as \x using \y+1] in {1,2,3} {\coordinate (side -\y) at ($(corner \y)-(corner \x)$);}
\draw (corner 1) -- (corner 2) -- (corner 3) -- (corner 4);
\draw (corner 1) --++(side 3) coordinate (corner 6) --++(side 2) coordinate (corner 5) --++(side 1);
\draw (corner 3) --++(side -1) --++(side 3);
\draw (corner 1) --++(side 2);
\end{tikzpicture}
\caption{The two rhombic tilings of $X(321)$. Each has three perimeter tiles, and none of the perimeter tiles of a given type are the same in the two tilings.}
\label{fig:321}
\end{figure}
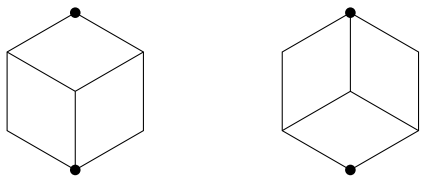

\begin{definition}
Fix a permutation $w$. If a specific tile $t$ appears as a right-perimeter tile among all tilings $T \in T(w)$, then $t$ is a \emph{forced right-perimeter} tile for $w$. Forced left-, top-, and bottom-perimeter tiles are defined analogously.
\end{definition}

As we saw in Figure~\ref{fig:321}, there are some $w$ for which no perimeter tiles are forced, but this is not always the case.

\begin{example}\label{ex:34251}
There are three rhombic tilings of $X(34251)$. As depicted in Figure~\ref{fig:34251}, each of the three has a right- and bottom-perimeter tile $\{1,5\}$. Therefore $\{1,5\}$ is a forced right- and bottom-perimeter tile for $34251$. There are no other forced perimeter tiles for this permutation.
\end{example}

\begin{figure}[htbp]
\begin{tikzpicture}
\begin{scope}
\clip (.2,2) rectangle (-2,-2);
\node[draw=none,minimum size=4cm,regular polygon,regular polygon sides=10] (a) {};
\end{scope}
\foreach \x in {1,2,3,4,5,6} {\coordinate (corner \x) at (a.side \x);}
\foreach \y [evaluate=\y as \x using \y+1] in {1,2,3,4,5} {\coordinate (side \y) at ($(corner \x)-(corner \y)$);}
\foreach \y [evaluate=\y as \x using \y+1] in {1,2,3,4,5} {\coordinate (side -\y) at ($(corner \y)-(corner \x)$);}
\fill[black!20] (corner 6) --++(side -1) --++(side -5) --++(side 1) --++(side 5);
\foreach \x in {(a.side 1),(a.side 6)} {\fill \x circle (2pt);}
\draw (corner 1) -- (corner 2) -- (corner 3) -- (corner 4) -- (corner 5) -- (corner 6);
\draw (corner 1) --++(side 3) coordinate (corner 10) --++(side 4) coordinate (corner 9) --++(side 2) coordinate (corner 8) --++(side 5) coordinate (corner 7) --++(side 1);
\draw (corner 5) -- (corner 8);
\draw (corner 1) --++(side 2) --++(side 3) --++(side -2);
\draw (corner 3) --++(side -1);
\draw (corner 4) --++(side -1) --++(side 4);
\end{tikzpicture}
\hspace{.5in}
\begin{tikzpicture}
\begin{scope}
\clip (.2,2) rectangle (-2,-2);
\node[draw=none,minimum size=4cm,regular polygon,regular polygon sides=10] (a) {};
\end{scope}
\foreach \x in {1,2,3,4,5,6} {\coordinate (corner \x) at (a.side \x);}
\foreach \y [evaluate=\y as \x using \y+1] in {1,2,3,4,5} {\coordinate (side \y) at ($(corner \x)-(corner \y)$);}
\foreach \y [evaluate=\y as \x using \y+1] in {1,2,3,4,5} {\coordinate (side -\y) at ($(corner \y)-(corner \x)$);}
\fill[black!20] (corner 6) --++(side -1) --++(side -5) --++(side 1) --++(side 5);
\foreach \x in {(a.side 1),(a.side 6)} {\fill \x circle (2pt);}
\draw (corner 1) -- (corner 2) -- (corner 3) -- (corner 4) -- (corner 5) -- (corner 6);
\draw (corner 1) --++(side 3) coordinate (corner 10) --++(side 4) coordinate (corner 9) --++(side 2) coordinate (corner 8) --++(side 5) coordinate (corner 7) --++(side 1);
\draw (corner 5) -- (corner 8);
\draw (corner 2) --++(side 3) --++(side -1);
\draw (corner 4) --++(side -2);
\draw (corner 4) --++(side -1) --++(side 4);
\draw (corner 10) --++(side 2);
\end{tikzpicture}
\hspace{.5in}
\begin{tikzpicture}
\begin{scope}
\clip (.2,2) rectangle (-2,-2);
\node[draw=none,minimum size=4cm,regular polygon,regular polygon sides=10] (a) {};
\end{scope}
\foreach \x in {1,2,3,4,5,6} {\coordinate (corner \x) at (a.side \x);}
\foreach \y [evaluate=\y as \x using \y+1] in {1,2,3,4,5} {\coordinate (side \y) at ($(corner \x)-(corner \y)$);}
\foreach \y [evaluate=\y as \x using \y+1] in {1,2,3,4,5} {\coordinate (side -\y) at ($(corner \y)-(corner \x)$);}
\fill[black!20] (corner 6) --++(side -1) --++(side -5) --++(side 1) --++(side 5);
\foreach \x in {(a.side 1),(a.side 6)} {\fill \x circle (2pt);}
\draw (corner 1) -- (corner 2) -- (corner 3) -- (corner 4) -- (corner 5) -- (corner 6);
\draw (corner 1) --++(side 3) coordinate (corner 10) --++(side 4) coordinate (corner 9) --++(side 2) coordinate (corner 8) --++(side 5) coordinate (corner 7) --++(side 1);
\draw (corner 5) -- (corner 8);
\draw (corner 2) --++(side 3) --++(side -1);
\draw (corner 4) --++(side -2) --++(side 4) --++(side -1);
\draw (corner 5) --++(side -2);
\end{tikzpicture}
\caption{The three rhombic tilings of $X(34251)$. The shaded right- and bottom-perimeter tile appears in all of them.}
\label{fig:34251}
\end{figure}
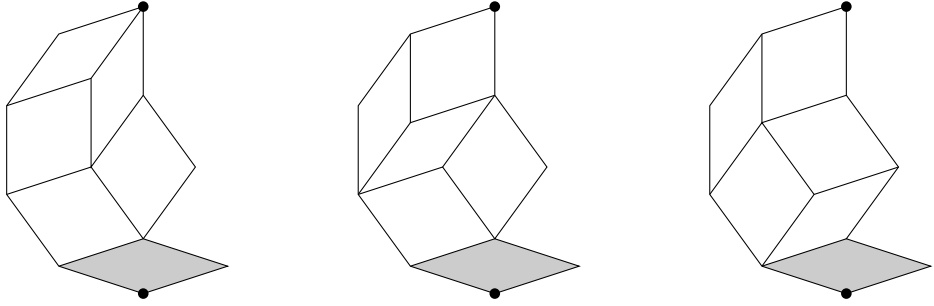

As we will see, $321$-patterns are critical to the determination of forced perimeter tiles, stemming from a previously obtained result.

\begin{definition}
Fix a permutation $w$ and a tiling $T \in T(w)$. This $T$ contains a \emph{subhexagon} if it has a configuration matching either tiling in Figure~\ref{fig:321}.
\end{definition}

\begin{proposition}[see {\cite[Theorem 6.4]{tenner rdpp}}]\label{prop:hexagon}
There is a tiling in $T(w)$ with a subhexagon having sides labeled $x < y < z$ if and only if $zyx$ is appears as a $321$-pattern in $w$.
\end{proposition}

We saw this demonstrated in Figure~\ref{fig:34251}.

\begin{example}
The permutation $34251$ has $321$-patterns $321$ and $421$. The first two tilings in Figure~\ref{fig:34251} have a subhexagon labeled $\{3,2,1\}$, and the last two tilings in Figure~\ref{fig:34251} have a subhexagon labeled $\{4,2,1\}$.
\end{example}

It follows from Elnitsky's work that if $w$ is $321$-avoiding, then $X(w)$ has exactly one rhombic tiling, meaning that all of its perimeter tiles are (trivially) forced. On the other hand, the converse does not hold: a permutation can have all of its perimeter tiles be forced and yet contain the pattern $321$.

\begin{example}
The permutation $3614725$ has a $321$-pattern and $|T(3614725)| = 2$, but all of its perimeter tiles are fixed. This is illustrated in Figure~\ref{fig:3614725}.
\end{example}

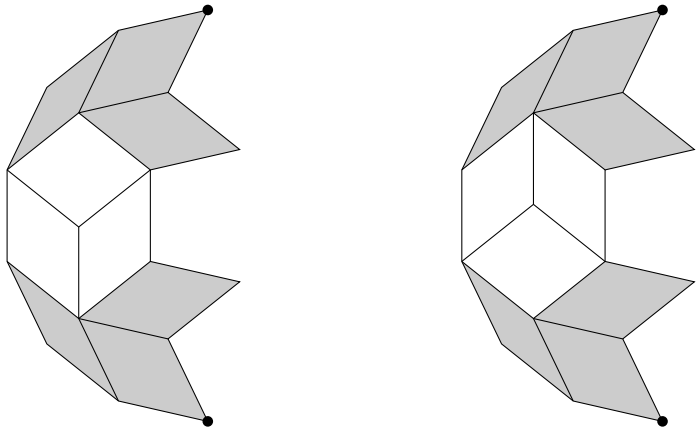
\begin{figure}[htbp]
\begin{tikzpicture}
\begin{scope}
\clip (.2,2.8) rectangle (-2.8,-2.8);
\node[draw=none,minimum size=5.6cm,regular polygon,regular polygon sides=14] (a) {};
\end{scope}
\foreach \x in {1,2,3,4,5,6,7,8} {\coordinate (corner \x) at (a.side \x);}
\foreach \y [evaluate=\y as \x using \y+1] in {1,2,3,4,5,6,7} {\coordinate (side \y) at ($(corner \x)-(corner \y)$);}
\foreach \y [evaluate=\y as \x using \y+1] in {1,2,3,4,5,6,7} {\coordinate (side -\y) at ($(corner \y)-(corner \x)$);}
\fill[black!20] (corner 1) --++(side 1) --++(side 2) --++(side 3) --++(side -2) --++(side 6) --++(side -1) --++(side -6) --++(side -3);
\fill[black!20] (corner 5) --++(side 5) --++(side 6) --++(side 7) --++(side -5) --++(side -2) --++(side -7) --++(side 2) --++(side -6);
\foreach \x in {(a.side 1),(a.side 8)} {\fill \x circle (2pt);}
\draw (corner 1) -- (corner 2) -- (corner 3) -- (corner 4) -- (corner 5) -- (corner 6) -- (corner 7) -- (corner 8);
\draw (corner 1) --++(side 3) coordinate (corner 14) --++(side 6) coordinate (corner 13) --++(side 1) coordinate (corner 12) --++(side 4) coordinate (corner 11) --++(side 7) coordinate (corner 10) --++(side 2) coordinate (corner 9) --++(side 5);
\draw (corner 2) --++(side 3) --++(side 6);
\draw (corner 4) --++(side -2) --++(side -1);
\draw (corner 5) --++(side 6) --++(side 7);
\draw (corner 7) --++(side -5) --++(side -2);
\draw (corner 4) --++(side 6) --++(side 4);
\draw (corner 12) --++(side 2);
\end{tikzpicture}
\hspace{1in}
\begin{tikzpicture}
\begin{scope}
\clip (.2,2.8) rectangle (-2.8,-2.8);
\node[draw=none,minimum size=5.6cm,regular polygon,regular polygon sides=14] (a) {};
\end{scope}
\foreach \x in {1,2,3,4,5,6,7,8} {\coordinate (corner \x) at (a.side \x);}
\foreach \y [evaluate=\y as \x using \y+1] in {1,2,3,4,5,6,7} {\coordinate (side \y) at ($(corner \x)-(corner \y)$);}
\foreach \y [evaluate=\y as \x using \y+1] in {1,2,3,4,5,6,7} {\coordinate (side -\y) at ($(corner \y)-(corner \x)$);}
\fill[black!20] (corner 1) --++(side 1) --++(side 2) --++(side 3) --++(side -2) --++(side 6) --++(side -1) --++(side -6) --++(side -3);
\fill[black!20] (corner 5) --++(side 5) --++(side 6) --++(side 7) --++(side -5) --++(side -2) --++(side -7) --++(side 2) --++(side -6);
\foreach \x in {(a.side 1),(a.side 8)} {\fill \x circle (2pt);}
\draw (corner 1) -- (corner 2) -- (corner 3) -- (corner 4) -- (corner 5) -- (corner 6) -- (corner 7) -- (corner 8);
\draw (corner 1) --++(side 3) coordinate (corner 14) --++(side 6) coordinate (corner 13) --++(side 1) coordinate (corner 12) --++(side 4) coordinate (corner 11) --++(side 7) coordinate (corner 10) --++(side 2) coordinate (corner 9) --++(side 5);
\draw (corner 2) --++(side 3) --++(side 6);
\draw (corner 4) --++(side -2) --++(side -1);
\draw (corner 5) --++(side 6) --++(side 7);
\draw (corner 7) --++(side -5) --++(side -2);
\draw (corner 5) --++(side -2) --++(side -4);
\draw (corner 11) --++(side -6);
\end{tikzpicture}
\caption{The two rhombic tilings of Elnitsky's polygon for the $321$-containing permutation $3614725$. The (identical) perimeter tiles in each tiling have been shaded.}
\label{fig:3614725}
\end{figure}

The relationship between forced perimeter tiles and $321$-patterns can also be phrased in terms of left-to-right and right-to-left maxima and minima. We will refer to these as \emph{LR-} or \emph{RL-max} or \emph{min}.

\begin{theorem}\label{thm:force right}
Fix a permutation $w = \cdots xy \cdots$ with $x > y$. There is a forced right-perimeter tile $\{x,y\}$ if and only if $x$ and $y$ do not appear in a $321$-pattern together; equivalently, if and only if $x$ is a LR-max and $y$ is a RL-min.
\end{theorem}

\begin{proof}
First suppose $\{x,y,z\}$ form a $321$-pattern in $w$. Then, by Proposition~\ref{prop:hexagon}, there is a tiling $T \in T(w)$ with a subhexagon whose edges are labeled $\{x,y,z\}$. As demonstrated in Figure~\ref{fig:321}, this hexagon will always have a tile $\{x,y\}$, but this is not always a right-perimeter tile of the hexagon. Then, by Corollary~\ref{cor:labels}, there exists a rhombic tiling of $X(w)$ in which the tile $\{x,y\}$ is not a right-perimeter tile.

Now suppose that $x = w(k)$ is a LR-max and $y = w(k+1)$ is a RL-min, and consider some $T \in T(w)$. The segment labeled $x$ along the rightside boundary of $X(w)$ is the edge of some $\{x,a\} \in T$. In order to fit inside of $X(w)$, we must have either $a > x$ or $a \le y$. Suppose that $a > x$. Because $x$ is a LR-max, we must have $a = w(h)$ for some $h > k$. But then a tile labeled $\{x,a\}$ would violate Corollary~\ref{cor:labels}. A similar argument shows that $a \not< y$. Therefore $a = y$, so the right-perimeter tile $\{x,y\}$ is forced.
\end{proof}

Reduced decompositions of $w$ and of its inverse $w^{-1}$ are left-to-right reflections of each other. Therefore, rhombic tilings of $X(w^{-1})$ can be obtained from rhombic tilings of $X(w)$.

\begin{lemma}[{cf. \cite{elnitsky}}]\label{lem:left and right symmetry}
Let $\tau : T(w) \rightarrow T(w^{-1})$ be the map that takes a left-to-right reflection of $T \in T(w)$ and deforms the edges so that the new leftside boundary is convex and all edges with the same label are made to be congruent. This $\tau$ is a bijection.
\end{lemma}

\begin{example}\label{ex:34251 and inverse}
The polygons $X(34251)$ and $X(53124)$, where $53124 = 34251^{-1}$, are depicted in Figure~\ref{fig:34251 and inverse}, along with a pair of corresponding tilings. 
\end{example}

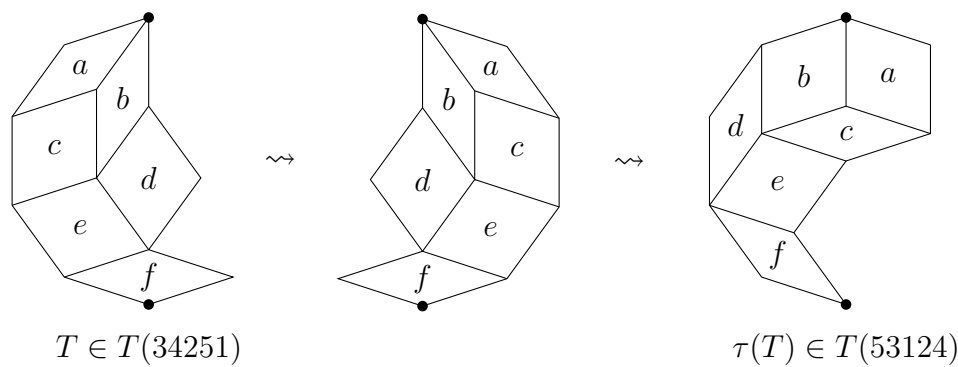
\begin{figure}[htbp]
\begin{tikzpicture}
\begin{scope}
\clip (.2,2) rectangle (-2,-2);
\node[draw=none,minimum size=4cm,regular polygon,regular polygon sides=10] (a) {};
\end{scope}
\foreach \x in {1,2,3,4,5,6} {\coordinate (corner \x) at (a.side \x);}
\foreach \y [evaluate=\y as \x using \y+1] in {1,2,3,4,5} {\coordinate (side \y) at ($(corner \x)-(corner \y)$);}
\foreach \y [evaluate=\y as \x using \y+1] in {1,2,3,4,5} {\coordinate (side -\y) at ($(corner \y)-(corner \x)$);}
\foreach \x in {(a.side 1),(a.side 6)} {\fill \x circle (2pt);}
\draw (corner 1) -- (corner 2) -- (corner 3) -- (corner 4) -- (corner 5) -- (corner 6);
\draw (corner 1) --++(side 3) coordinate (corner 10) --++(side 4) coordinate (corner 9) --++(side 2) coordinate (corner 8) --++(side 5) coordinate (corner 7) --++(side 1);
\draw (corner 5) -- (corner 8);
\draw (corner 1) --++(side 2) --++(side 3) coordinate (corner a)--++(side -2);
\draw (corner 3) --++(side -1);
\draw (corner 4) --++(side -1) --++(side 4);
\draw (1.75,0) node {$\rightsquigarrow$};
\draw (0,-2.5) node {$T \in T(34251)$};
\draw ($(corner 1)!.5!(corner 3)$) node {$a$};
\draw ($(corner 1)!.5!(corner a)$) node {$b$};
\draw ($(corner 3)!.5!(corner a)$) node {$c$};
\draw ($(corner 8)!.5!(corner 10)$) node {$d$};
\draw ($(corner 4)!.5!(corner 8)$) node {$e$};
\draw ($(corner 5)!.5!(corner 7)$) node {$f$};
\end{tikzpicture}
\begin{tikzpicture}
\begin{scope}
\clip (-.2,2) rectangle (2,-2);
\node[draw=none,minimum size=4cm,regular polygon,regular polygon sides=10] (a) {};
\end{scope}
\foreach \x in {1,10,9,8,7,6} {\coordinate (corner \x) at (a.side \x);}
\foreach \y [evaluate=\y as \x using \y-1] in {6,7,8,9,10} {\coordinate (side \y) at ($(corner \x)-(corner \y)$);}
\foreach \y [evaluate=\y as \x using \y-1] in {6,7,8,9,10} {\coordinate (side -\y) at ($(corner \y)-(corner \x)$);}
\foreach \x in {(a.side 1),(a.side 6)} {\fill \x circle (2pt);}
\draw (corner 1) -- (corner 10) -- (corner 9) -- (corner 8) -- (corner 7) -- (corner 6);
\draw (corner 1) --++(side 3) coordinate (corner 2) --++(side 2) coordinate (corner 3) --++(side 4) coordinate (corner 4) --++(side 1) coordinate (corner 5) --++(side 5);
\draw (corner 7) -- (corner 4);
\draw (corner 1) --++(side 4) --++(side 3) coordinate (corner a) --++(side -4);
\draw (corner 9) --++(side -5);
\draw (corner 8) --++(side -5) --++(side 2);
\draw (2.75,0) node {$\rightsquigarrow$};
\draw[white] (0,-2.5) node {rhombic tiling};
\draw ($(corner 1)!.5!(corner 9)$) node {$a$};
\draw ($(corner 1)!.5!(corner a)$) node {$b$};
\draw ($(corner a)!.5!(corner 9)$) node {$c$};
\draw ($(corner 3)!.5!(corner a)$) node {$d$};
\draw ($(corner 4)!.5!(corner 8)$) node {$e$};
\draw ($(corner 5)!.5!(corner 7)$) node {$f$};
\end{tikzpicture}
\hspace{.1in}
\begin{tikzpicture}
\begin{scope}
\clip (.2,2) rectangle (-2,-2);
\node[draw=none,minimum size=4cm,regular polygon,regular polygon sides=10] (a) {};
\end{scope}
\foreach \x in {1,2,3,4,5,6} {\coordinate (corner \x) at (a.side \x);}
\foreach \y [evaluate=\y as \x using \y+1] in {1,2,3,4,5} {\coordinate (side \y) at ($(corner \x)-(corner \y)$);}
\foreach \y [evaluate=\y as \x using \y+1] in {1,2,3,4,5} {\coordinate (side -\y) at ($(corner \y)-(corner \x)$);}
\foreach \x in {(a.side 1),(a.side 6)} {\fill \x circle (2pt);}
\draw (corner 1) -- (corner 2) -- (corner 3) -- (corner 4) -- (corner 5) -- (corner 6);
\draw (corner 2) --++(side 3) coordinate (corner a) --++(side 2);
\draw (corner 1) --++(side 3) --++(side 5) coordinate (corner 9)--++(side -3) (corner 10)--++(side -5);
\draw (corner a) --++(side -1);
\draw (corner 4) --++(side 5) coordinate (corner 7) --++(side 4);
\draw (corner a) --++(side 5) coordinate (corner 8) coordinate --++(side 2);
\draw (corner 8) --++(side -1);
\draw (0,-2.5) node {$\tau(T) \in T(53124)$};
\draw ($(corner 1)!.5!(corner 9)$) node {$a$};
\draw ($(corner 1)!.5!(corner a)$) node {$b$};
\draw ($(corner a)!.5!(corner 9)$) node {$c$};
\draw ($(corner 2)!.5!(corner 4)$) node {$d$};
\draw ($(corner 4)!.5!(corner 8)$) node {$e$};
\draw ($(corner 5)!.5!(corner 7)$) node {$f$};
\end{tikzpicture}
\caption{Rhombic tilings of Elnitsky polygons for $w$ and $w^{-1}$ are related by reflection and deformation. To clarify the transformation in this example, the tiles in the first figure have been labeled and their images have been labeled correspondingly in the second and third figure.}
\label{fig:34251 and inverse}
\end{figure}

We can combine Theorem~\ref{thm:force right} and Lemma~\ref{lem:left and right symmetry} to characterize forced left-perimeter tiles.

\begin{corollary}\label{cor:force left}
Fix a permutation $w$. There is a forced left-perimeter tile $\{k,k+1\}$ if and only if $w^{-1}(k) > w^{-1}(k+1)$, and $k$ and $k+1$ do not appear in a $321$-pattern together. This is equivalent to $k+1$ being a LR-max in $w$ and $k$ being a RL-min.
\end{corollary}

\begin{proof}
By Lemma~\ref{lem:left and right symmetry}, forcing the left-perimeter tile $\{k,k+1\}$ for $w$ is equivalent to forcing the right-perimeter tile $\{w^{-1}(k),w^{-1}(k+1)\}$ in $w^{-1}$. By Theorem~\ref{thm:force right}, this is equivalent to $w^{-1}(k)$ and $w^{-1}(k+1)$ not appearing in a $321$-pattern together in $w^{-1}$, which is equivalent to $k$ and $k+1$ not appearing in a $321$-pattern together in $w$.
\end{proof}

Forcing top- and bottom-perimeter tiles has a similar flavor to Theorem~\ref{thm:force right} and Corollary~\ref{cor:force left}.

\begin{theorem}\label{thm:force top}
Fix a permutation $w$. There is a forced top-perimeter tile, necessarily $\{1,w(1)\}$, if and only if $w(1)$ and $1$ do not appear in a $321$-pattern together; equivalently, if and only if the first LR-min after $w(1)$ is $1$.
\end{theorem}

\begin{proof}
First suppose that $\{w(1), k, 1\}$ is a $321$-pattern in $w$. Then, as before, Proposition~\ref{prop:hexagon} means that there is a rhombic tiling of $X(w)$ with a subhexagon whose edges are labeled $\{w(1),k,1\}$, in which the tile $\{1,w(1)\}$ is not a top-perimeter tile. Thus, no matter where this hexagon is positioned in the polygon $X(w)$, the resulting tiling of $X(w)$ has no top-perimeter tile.

Now suppose that the first LR-min after $w(1)$ is $1$, and consider some $T \in T(w)$. The segment labeled $w(1)$ along the rightside boundary of $X(w)$ is the edge of some $\{w(1),a\} \in T$. By Corollary~\ref{cor:labels}, we must have $a < w(1)$. If $a > 1$ then, because $1$ is the first LR-min after $w(1)$, we must have a tile $\{1,a\}$ in $T$. However, this would contradict Corollary~\ref{cor:labels}, so $a = 1$. Therefore the top-perimeter tile $\{w(1),1\}$ is forced.
\end{proof}

The conditions for a forced bottom-perimeter tile are analogous.

\begin{corollary}\label{cor:force bottom}
Fix a permutation $w \in \mf{S}_n$. There is a forced bottom-perimeter tile, necessarily $\{n,w(n)\}$, if and only if $w(n)$ and $n$ do not appear in a $321$-pattern together; equivalently, if and only if the first RL-max after $w(n)$ is $n$.
\end{corollary}

There is a relationship between right-/left- and top-/bottom-perimeter tiles.

\begin{corollary}\label{cor:right at top means top}
Fix a permutation $w \in \mf{S}_n$. If there is a forced right-perimeter tile $\{w(1),w(2)\}$, then this is a forced top-perimeter tile, too. If there is a forced left-perimeter tile $\{1,2\}$, then this is a forced top-perimeter tile, too. If there is a forced right-perimeter tile $\{w(n-1),w(n)\}$, then this a forced bottom-perimeter tile, too. If there is a forced left-perimeter tile $\{n-1,n\}$, then this is a forced bottom-perimeter tile, too.
\end{corollary}

\begin{proof}
Suppose that $\{w(1),w(2)\}$ is a forced right-perimeter tile. Then, by Theorem~\ref{thm:force right}, $w(2) = 1$. Then, by Theorem~\ref{thm:force top}, $\{w(1),w(2) = 1\}$ is a forced top-perimeter tile. Suppose that $\{1,2\}$ is a forced left-perimeter tile. Then, by Corollary~\ref{cor:force left}, $w(1) = 2$, and so by Theorem~\ref{thm:force top}, $\{1,2\}$ is a forced top-perimeter tile. The latter parts of the result follow by symmetry.
\end{proof}

The converses to the statements in Corollary~\ref{cor:right at top means top} do not hold.

\begin{example}
Consider $w = 2341$. This permutation has a forced top-perimeter tile $\{1,2\}$, but $\{2,3\}$ is not a forced right-perimeter tile.
\end{example}

\section{Optimally forced perimeter}\label{sec:optimizing}

Having characterized how perimeter tiles can be forced in $T(w)$, we close by looking at an optimization problem. To be specific, which permutations have maximally many forced right-perimeter tiles, and how many such permutations exist in $\mf{S}_n$? Lemma~\ref{lem:left and right symmetry} means that analogous questions about left-perimeter tiles can be understood through the answers to these questions.

Consider $w \in \mf{S}_n$. To have maximally many forced right-perimeter tiles would mean $\lfloor n/2 \rfloor$ such tiles. Therefore we will assume $n = 2m$.

\begin{theorem}\label{thm:descents in optimal}
A permutation $w \in \mf{S}_{2m}$ has maximally many forced right-perimeter tiles, if and only if $w(2k-1) > w(2k)$ for all $k$, and
\begin{equation}\label{eqn:alternating}
w(1) < w(3) < \cdots < w(2m-1) \hspace{.25in} \text{and} \hspace{.25in} w(2) < w(4) < \cdots < w(2m).
\end{equation}
\end{theorem}

\begin{proof}
Suppose that $w$ has maximally many forced right-perimeter tiles. The requirement that $w(2k-1) > w(2k)$ for all $k$ follows from Corollary~\ref{cor:labels}, and the inequalities listed in \eqref{eqn:alternating} follow from Theorem~\ref{thm:force right}.

Now suppose that $w(2k-1) > w(2k)$ for all $k$, and that the inequalities listed in \eqref{eqn:alternating} hold. Corollary~\ref{cor:labels} and Theorem~\ref{thm:force right} mean that there are forced right-perimeter tiles $\{w(2k-1),w(2k)\}$ for all $k$.
\end{proof}

We demonstrate Theorem~\ref{thm:descents in optimal} with two examples.

\begin{example}
The permutation $315264$ satisfies the requirements of Theorem~\ref{thm:descents in optimal}, and hence it has three forced right-perimeter tiles. The permutation $325164$, on the other hand, does not, because $2 \not<1 < 4$. Figure~\ref{fig:optimizing examples} depicts the only rhombic tiling of $X(315264)$, in which there are three (forced) right-perimeter tiles, and a rhombic tiling of $X(325164)$ in which there are only two right-perimeter tiles.
\end{example}

\begin{figure}[htbp]
\begin{tikzpicture}
\begin{scope}
\clip (.2,2.4) rectangle (-2.4,-2.4);
\node[draw=none,minimum size=4.8cm,regular polygon,regular polygon sides=12] (a) {};
\end{scope}
\foreach \x in {1,2,3,4,5,6,7} {\coordinate (corner \x) at (a.side \x);}
\foreach \y [evaluate=\y as \x using \y+1] in {1,2,3,4,5,6} {\coordinate (side \y) at ($(corner \x)-(corner \y)$);}
\foreach \y [evaluate=\y as \x using \y+1] in {1,2,3,4,5,6} {\coordinate (side -\y) at ($(corner \y)-(corner \x)$);}
\foreach \x in {(a.side 1),(a.side 7)} {\fill \x circle (2pt);}
\draw (corner 1) -- (corner 2) -- (corner 3) -- (corner 4) -- (corner 5) -- (corner 6) -- (corner 7);
\draw (corner 1) --++(side 3) coordinate (corner 12) --++(side 1) coordinate (corner 11) --++(side 5) coordinate (corner 10) --++(side 2) coordinate (corner 9) --++(side 6) coordinate (corner 8) --++(side 4);
\draw (corner 2) --++(side 3) --++(side 2) --++(side 5) --++(side 4);
\draw (-.5,-3) node {The only rhombic};
\draw (-.5,-3.5) node {tiling of $X(315264)$};
\end{tikzpicture}
\hspace{1in}
\begin{tikzpicture}
\begin{scope}
\clip (.2,2.4) rectangle (-2.4,-2.4);
\node[draw=none,minimum size=4.8cm,regular polygon,regular polygon sides=12] (a) {};
\end{scope}
\foreach \x in {1,2,3,4,5,6,7} {\coordinate (corner \x) at (a.side \x);}
\foreach \y [evaluate=\y as \x using \y+1] in {1,2,3,4,5,6} {\coordinate (side \y) at ($(corner \x)-(corner \y)$);}
\foreach \y [evaluate=\y as \x using \y+1] in {1,2,3,4,5,6} {\coordinate (side -\y) at ($(corner \y)-(corner \x)$);}
\foreach \x in {(a.side 1),(a.side 7)} {\fill \x circle (2pt);}
\draw (corner 1) -- (corner 2) -- (corner 3) -- (corner 4) -- (corner 5) -- (corner 6) -- (corner 7);
\draw (corner 1) --++(side 3) coordinate (corner 12) --++(side 2) coordinate (corner 11) --++(side 5) coordinate (corner 10) --++(side 1) coordinate (corner 9) --++(side 6) coordinate (corner 8) --++(side 4);
\draw (corner 6) --++(side -4) --++(side -5) --++(side -1);
\draw (corner 4) --++(side -2) --++(side -1);
\draw (corner 2) --++(side 3);
\draw (-.5,-3) node {A rhombic};
\draw (-.5,-3.5) node {tiling of $X(325164)$};
\end{tikzpicture}
\caption{The permutation $315264$ has three forced right-perimeter tiles, whereas the permutation $325164$ does not.}
\label{fig:optimizing examples}
\end{figure}
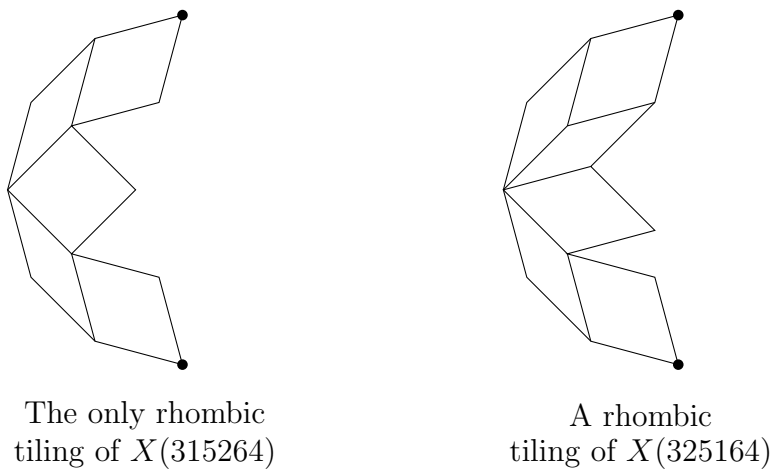

By Theorem~\ref{thm:descents in optimal}, a permutation $w$ with maximally many forced right-perimeter tiles is $321$-avoiding. Therefore, as discussed in Section~\ref{sec:forced}, there is only one rhombic tiling of $X(w)$ for such a permutation, so, in fact, all tiles are forced. Conversely, not all $321$-avoiding permutations satsify the requirements of Theorem~\ref{thm:descents in optimal}. For example, $3412 \in \mf{S}_4$ avoids $321$, and hence $X(3412)$ has a single rhombic tiling, but it has only one forced right-perimeter tile.

Recall that we are only considering permutations with full support. Thus the characterization in Theorem~\ref{thm:descents in optimal} says that a fully supported permutation $w \in \mf{S}_{2m}$ has maximally many forced right-perimeter tiles if and only if $w$ is a $321$-avoiding alternating permutation. Apart from the requirement to be fully supported, this characterization might call the Catalan numbers to mind.

\begin{lemma}[\!\!{\cite[Item 146]{stanley catalan}}]\label{lem:stanley catalan}
The $321$-avoiding alternating permutations in $\mf{S}_{2m}$ are enumerated by the Catalan number $C_m$.
\end{lemma}

In fact, we can enumerate fully supported $321$-avoiding alternating permutations in $\mf{S}_{2m}$ via a bijection with $321$-avoiding alternating permutations in $\mf{S}_{2m-2}$.

Define a map $\phi:\{321\text{-avoiding alternating permutations in } \mf{S}_{2m-2}\} \rightarrow \mf{S}_{2m}$ by
$$\big(\phi(v)\big)(i) = \begin{cases}
v(i) + 1 & \text{if $i<2m-1$ is odd},\\
v(i-2) + 1 & \text{if $i>2$ is even},\\
1 & \text{if $i = 2$}, \text{ and}\\
2n & \text{if $i = 2m-1$.}
\end{cases}$$
We will show, through a sequence of lemmas, that the image of $\phi$ is a subset of the set of fully supported $321$-avoiding alternating permutations in $\mf{S}_{2m}$, and finally that $\phi$ is a bijection onto this set.

\begin{lemma}\label{lem:phi fully supported}
The permutation $\phi(v)$ is fully supported.
\end{lemma}

\begin{proof}
Fix $v \in \{321\text{-avoiding alternating permutations in } \mf{S}_{2m-2}\}$ and set $w := \phi(v)$. Suppose that $\{w(1),\ldots, w(k)\} = \{1,\ldots,k\}$ for some $k < 2m$. Consider first the case that $k$ is even. Then
$$\{v(1) + 1, 1, v(3) + 1, v(2) + 1, \ldots, v(k-1) + 1, v(k-2) + 1\} = \{1,\ldots, k\},$$
and so $\{v(1),\ldots,v(k-1)\} = \{1,\ldots, k-1\}$. Because $k$ is even and $v$ is alternating, we must have $v(k-1) > v(k)$, so the first $k-1$ positions of $v$ cannot hold the $k-1$ smallest values, which is a contradiction.

Now consider the case that $k$ is odd. Because $w(2m-1) = 2m$, it must be that $k < 2m-1$. Then we have
$$\{v(1) + 1, 1, v(3) + 1, v(2) + 1, \ldots, v(k-2) + 1, v(k-3) + 1, v(k) + 1\} = \{1,\ldots, k\},$$
and so $\{v(1),\ldots, v(k-2), v(k)\} = \{1,\ldots, k-1\}$. Because $k$ is odd and $v$ is alternating, we must have that $v(k-1)$ is less than both $v(k-2)$ and $v(k)$, which is impossible if $\{v(1),\ldots, v(k-2), v(k)\}$ are the smallest values in the permutation.

Therefore $w$ is fully supported.
\end{proof}

\begin{lemma}\label{lem:phi alternating}
The permutation $\phi(v)$ is alternating.
\end{lemma}

\begin{proof}
Fix $v \in \{321\text{-avoiding alternating permutations in } \mf{S}_{2m-2}\}$ and set $w := \phi(v)$. Because $v$ is alternating and avoids $321$, we must have $v(i) < v(i+2)$ for all $i$. Therefore, for all $k$,
\begin{align*}
w(2k+1) &= v(2k+1) + 1,\\
w(2k) &= v(2k-2) + 1 < v(2k) + 1,\text{ and}\\
w(2k+2) &= v(2k) + 1 < v(2k+2) + 1.
\end{align*}
As an alternating permutation, $v(1)> v(w)$ and $v(2k+1) > v(2k), v(2k+2)$ for all $k\ge 1$. Therefore $w(1) > w(2)$ and $w(2k+1) > w(2k), w(2k+2)$ for all $k \ge 1$, as well.
\end{proof}

\begin{lemma}\label{lem:phi 321-avoiding}
The permutation $\phi(v)$ is $321$-avoiding.
\end{lemma}

\begin{proof}
In an alternating permutation, being $321$-avoiding is equivalent to satisfying $w(k) < w(k+2)$ for all $k$. Fix $v \in \{321\text{-avoiding alternating permutations in } \mf{S}_{2m-2}\}$ and set $w := \phi(v)$. By Lemma~\ref{lem:phi alternating}, $w$ is an alternating permutation.

Because $v$ is a $321$-avoiding alternating permutation, we have $v(k) < v(k+2)$ for all $k$. In particular,
$$v(1) + 1 < v(3) + 1 < \cdots < v(2m-3) + 1 \le 2m-2 + 1 < 2m$$
and
$$1 < v(2) + 1 < v(4) + 1 < \cdots < v(2m-4) + 1 < v(2m-2) + 1.$$
Therefore, by definition of $\phi$, we have $w(k) < w(k+2)$ for all $k$, and so $w$ is $321$-avoiding.
\end{proof}

From these results, we see that the image of $\phi$ is actually a subset of the fully supported $321$-avoiding alternating permutations in $\mf{S}_{2m}$. In fact, we can say much more.

\begin{theorem}\label{thm:optimal right enumeration}
The map $\phi$ is a bijection from
$$\{321\text{-avoiding alternating permutations in } \mf{S}_{2m-2}\}$$
to 
$$\{\text{fully supported } 321\text{-avoiding alternating permutations in } \mf{S}_{2m}\},$$
and the number of $w \in \mf{S}_{2m}$ with $m$ forced right-perimeter tiles (equivalently, the number of fully supported $321$-avoiding alternating permutations) is the Catalan number $C_{m-1}$.
\end{theorem}

\begin{proof}
By Lemmas~\ref{lem:phi fully supported}, \ref{lem:phi alternating}, and~\ref{lem:phi 321-avoiding}, we have 
\begin{equation}\label{eqn:phi bijection}
\begin{split}
\phi: \{&\text{$321$-avoiding alternating permutations in } \mf{S}_{2m-2}\}\\
&\rightarrow \{\text{fully supported $321$-avoiding alternating permutations in } \mf{S}_{2m}\}.
\end{split}
\end{equation}
In fact, $\phi$ is a bijection: the preimage of an arbitrary fully supported $321$-avoiding alternating permutation $w \in \mf{S}_{2m}$ is the permutation $v \in \mf{S}_{2m-2}$ defined by
$$v(i) = \begin{cases}
w(i) - 1 & \text{if $i$ is odd, and}\\
w(i+2) - 1 & \text{if $i$ is even.}
\end{cases}$$
That this $v$ is alternating and $321$-avoiding follows from the same types of arguments presented in the proofs of Lemmas~\ref{lem:phi alternating} and~\ref{lem:phi 321-avoiding}. Thus, the two sets in \eqref{eqn:phi bijection} are equinumerous, and both enumerated by $C_{m-1}$, thanks to Lemma~\ref{lem:stanley catalan}.
\end{proof}

We give an example of this enumeration for $m=4$, including a demonstration on the map $\phi$.

\begin{example}
The $C_3 = 5$ fully supported $321$-avoiding alternating permutations in $\mf{S}_8$ are listed in Table~\ref{table:alternating example}, along with their corresponding (via $\phi$) $321$-avoiding alternating permutations in $\mf{S}_6$.
\end{example}

\begin{table}[htbp]
$$\begin{array}{c|c|c}
\text{$321$-avoiding alternating} & & \text{fully supported $321$-avoiding}\\
\text{permutation in $\mf{S}_6$} & \phi & \text{alternating permutation in $\mf{S}_8$}\\
\hline
214365 & \mapsto & 31527486\\
215364 & \mapsto & 31627485\\
314265 & \mapsto & 41527386\\
315264 & \mapsto & 41627385\\
415263 & \mapsto & 51627384
\end{array}$$
\caption{There are $C_3 = 5$ permutations in $\mf{S}_6$ that are $321$-avoiding and alternating. These are in bijection with the fully supported $321$-avoiding alternating permutations in $\mf{S}_8$, by means of the map $\phi$.}
\label{table:alternating example}
\end{table}

\section{Directions for further research}\label{sec:further}

We have focused on Elnitksy polygons because of the Coxeter-theoretic significance of their tiles. Of course, the notion of forced perimeter tiles exists beyond the context of Elnitsky polygons and rhombic tilings, and this deserves attention. Certain regions---Elnitsky polygons or otherwise---may have forced non-perimeter tiles, as well. One could also relax the definition of forcing to study when a given tile is ``$\alpha$-forced,'' where $\alpha \in [0,1]$ is the proportion of tilings in which the tile appears, and so a forced tile in this paper would be a $1$-forced tile in that context.

\end{document}